\documentclass[11pt,a4paper]{article}
\usepackage[top=3cm, bottom=3cm, left=3cm, right=3cm]{geometry}
\usepackage{tabu}
\usepackage{breqn}
\usepackage[hang,flushmargin]{footmisc} 

\usepackage{amsmath,amsthm,amssymb,graphicx, multicol, array}
\usepackage{enumerate}
\usepackage{enumitem}
\usepackage{hyperref}
\usepackage{setspace}
\usepackage{xcolor}
\usepackage{currfile}

\usepackage{tabto}

\newtheorem{thm}{Theorem}[section]
\newtheorem{lem}{Lemma}[section]
\newtheorem{conj}{Conjecture}[section]
\newtheorem{exa}{Example}[section]
\newtheorem{cor}{Corollary}[section]
\newtheorem{dfn}{Definition}[section]

\newcommand{\N}{\mathbb{N}}
\newcommand{\Z}{\mathbb{Z}}
\newcommand{\Q}{\mathbb{Q}}
\newcommand{\R}{\mathbb{R}}
\newcommand{\C}{\mathbb{C}}

\newcommand{\tP}{\mathbb{P}}

\usepackage{fancyhdr}
\title{Approximate Atkin-Serre Conjecture}
\date{}
\author{N. A. Carella}

\begin{document}
\thispagestyle{empty}
\date{}

\maketitle
\textbf{\textit{Abstract}:} Let $\lambda(n)$ be the $n$th coefficient of a modular form $f(z)=\sum_{n\geq 1} \lambda(n)q^n$ of 
weight $k\geq 4$, let $p^m$ be a prime power, and let $\varepsilon>0$ be a small number. A pair of completely different 
approximations of the Atkin-Serre conjecture are presented in this note. The first approximation of the lower bound
$p^{(k-1)m/2-2k+2\varepsilon}\leq   \left | \lambda(p^m)\right |$ is true for all sufficiently large prime powers, and the second 
approximation of the lower bound $n^{(k-3)/2+\log \log  \log n/\log \log n}\leq   \left | \lambda(n)\right |$ is true on a subset of 
integers of density $1$.

\let\thefootnote\relax\footnote{ \today \date{} 
}

{\let\bfseries\mdseries
\tableofcontents
}
\section{Introduction}\label{S5599}
The properties of the \textit{Fourier coefficients} $\lambda: \N\longrightarrow \C$ of modular forms
\begin{equation} \label{eq5599.035}
f(s)=\sum_{n \geq1} \lambda(n)q^n=\lambda(1)q+\lambda(2)q^2+\lambda(3)q^3+\lambda(4)q^4+\cdots,
\end{equation}
where $s\in \mathbb{C}$ is a complex number in the upper half plane, and $q=e^{i2\pi s }$, are the topics of many studies. Basic information on modular forms, classified by various parameters such as \textit{level} $N\geq1$, \textit{weight} $k\geq1$, et cetera, and other data are archived in \cite{LMFDB}. The corresponding $L$-function $L(s,f)=\sum_{n \geq 1}\lambda(n)n^{-s}$ is analytic on the complex half plane $\mathcal{H}_f=\{s\in \C: \Re e(s)>(k+1)/2\}$, and its functional equation 
\begin{equation} \label{eq5599.040}
\xi(s)=\left ( 2 \pi\right )^{-s}\Gamma(s)L(s,f), \qquad \xi(s)=\xi(k-s)
\end{equation}
facilitates an analytic continuation to the entire complex plane, see \cite[p.\ 376]{HI1989}. One of the basic property of the Fourier coefficients is the \textit{dynamic range} of its magnitude

\begin{equation} \label{eq5599.045}
 -L \leq \lambda\left (n\right )\leq U,
\end{equation}
where $L> 0$ and $U> 0$ are the lower bounds and upper bounds respectively. The earliest results for the upper bounds seems to be the Hecke estimate
\begin{equation} \label{eq5599.055}
\left |\lambda\left (p^m\right )\right | \leq cp^{(k-1)m},
\end{equation} 
where $p^m$ is a prime power, $c>0$ is a constant, and $k\geq 1$ is the weight, see \cite[Theorem 4]{ZD2013}, \cite[Proposition 5.4]{CH2019}, et cetera. After many partial results by many authors, this line of research culminated with the effective upper bound (known as Deligne theorem)
\begin{equation} \label{eq5599.100}
\left |\lambda\left (n\right )\right | \leq d(n)n^{(k-1)/2+\varepsilon},
\end{equation} 
where $d(n)$ is the divisors function, and $\varepsilon>0$ is a small number. Furthermore, the partial upper bound
\begin{equation} \label{eq5599.110}
\left |\lambda\left (n\right )\right | \leq 2n^{(k-1)/2}\left (\log n\right )^{-1/2+o(1)},
\end{equation}
on a subset of integers of density $1$, was proved in \cite[Theorem 1.1]{LS2019}. \\

On the other hand, the lower bounds for nonvanishing Fourier coefficients have no effective results. However, there is a claims that specifies an effective lower bound.

\begin{conj} {\normalfont (Atkin-Serre Conjecture)} Let $f$ be a non-CM modular form of weight $k\geq4$. If $p$ is sufficiently large prime, then, for each $\varepsilon> 0$, there exist constants $c(\varepsilon,f)>0$, such that
\begin{equation} \label{eq5599.115}
\left |\lambda\left (p\right )\right | \geq c(\varepsilon,f)p^{(k-3)/2-\varepsilon}.
\end{equation}  
\end{conj}
Various new partial results on this conjecture appear in \cite{RJ2005}, \cite{GT2020}, et alii. Furthermore, the effective lower bound
\begin{equation} \label{eq5599.140}
2p^{(k-1)n/2}\frac{\log \log p^n}{(\log p^n )^{1/2} }< \left|\lambda\left (p^n\right )\right |,
\end{equation}
on a subset of primes of density $1$, was proved in \cite[Theorem 1]{GT2020}. A much weaker, almost trivial, but unconditional lower bound proved in \cite[Theorem 1]{MS1987}, has the form  
\begin{equation} \label{eq5599.120}
\left |\lambda\left (n\right )\right | \geq (\log n)^c,
\end{equation} 
where $n$ is an integer for which $\lambda(n)$ is odd, and $c>0$ is an effectively computable absolute constant. The authors commented, [op. cit., page 393], that an application of Roth theorem for the approximations of algebraic integers, see \eqref{eq2020.090}, to the Ramanujan tau function $\tau(p^n)$ yields
\begin{equation} \label{eq5599.030}
\left |\tau\left (p^n\right )\right | \gg p^{11(n-4)/2-\varepsilon},
\end{equation}  
where the implied constant depends on the prime power $p^n\geq1$, and $\varepsilon >0$. But the implied constant is not computable. \\

The proof of the first approximation in Section \ref{S9955}, applies an explicit version of Liouville theorem for the approximations of algebraic integers to obtain the following result.
\begin{thm} \label{thm5599.200} Let $\lambda(m)\ne0 $ be the $m$th coefficient of a modular form $f(z)=\sum_{m\geq 1} \lambda(m)q^m$ of 
weight $k\geq 4$. If the integer $m=p^n$ is a prime power, and $\varepsilon>0$, then
\begin{equation} \label{eq5599.200}
\left |\lambda\left (p^n\right )\right | \geq \frac{1}{8}p^{(k-3)n/2-2k+2-\varepsilon},
\end{equation} 
as $p^{n/2}\to \infty.$
\end{thm}
The parameters $p^{n/2}> p^{5}$, and $k\geq 4$ produce a nontrivial lower bound. 

\begin{exa}\label{exa5599.500} {\normalfont For $k=4$, consider the trace of the Frobenius $a(n)$ of the modular form  $f(z)=\sum_{n\geq1}a(n)q^n$ attached to a nonsingular elliptic curve. Suppose that $a(p)\ne0$. Then, at the prime power $p^n$, it satisfies the new explicit lower bound
\begin{equation} \label{eq5599.510}
\left |a \left (p^n\right )\right | \geq \frac{1}{8} p^{n/2-6-\varepsilon},
\end{equation}
as $ p^5\leq p^{n/2}\to \infty.$
}
\end{exa}
\begin{exa}\label{exa5599.550} {\normalfont For $k= 12$, consider the Ramanujan tau function $\tau(n)$ associated with the modular form  $f(z)=\sum_{n\geq1}\tau(n)q^n$. Suppose that $\tau(p)\ne0$. Then, at the prime power $p^n$, it satisfies has the new explicit lower bound
\begin{equation} \label{eq5599.550}
\left |\tau\left (p^n\right )\right | \geq \frac{1}{8} p^{9n/2-22-\varepsilon},
\end{equation}
as $ p^5\leq p^{n/2}\to \infty.$
}
\end{exa}

The proof of the second approximation in Section \ref{S9911}, employs completely different method to derive an explicit lower bound for almost all the coefficients. This result has the following claim.
\begin{thm} \label{thm9911.200} The $n$th Fourier coefficient of a modular form  $f(z)=\sum_{n\geq1}\lambda(n)q^n$ of weight $k\geq 4$ 
satisfies the followings lower bound and upper bound
\begin{equation}\label{eq9911.200}
n^{\frac{k-3}{2}+\frac{\log \log \log n}{\log \log n } }\leq   \left | \lambda(n)\right |\leq n^{\frac{k-1}{2}+\varepsilon}\nonumber,
\end{equation}
where $\varepsilon>0$ is a small number, on a subset of integers $\mathcal{N}_f=\{n\in\N: \lambda(n)\ne0\}$ of density $\delta (\mathcal{N}_f)=1$.
\end{thm}
Consequently, it confirms the Atkin-Serre conjecture for almost every integer. A similar result, in \cite[Theorem 1.6]{HS2021}, proves an asymptotic formula for the number of exceptions.

\begin{exa}\label{exa5599.600} {\normalfont For $k=4$, consider the trace of the Frobenius $a(n)$ of the modular form  $f(z)=\sum_{n\geq1}a(n)q^n$ attached to a nonsingular elliptic curve. Let $r=rad(n)$ be the radical of $n\geq1$, and suppose that $a(r)\ne0$. Then, it satisfies the new explicit lower bound
\begin{equation} \label{eq5599.610}
\left |a \left (n\right )\right | \geq  n^{\frac{1}{2}+\frac{\log \log \log n}{\log \log n } },
\end{equation}
on a subset of integers $\mathcal{N}_f=\{n\in\N: \lambda(n)\ne0\}$ of density $\delta (\mathcal{N}_f)=1$.
}
\end{exa}
\begin{exa}\label{exa5599.650} {\normalfont For $k= 12$, consider the Ramanujan tau function $\tau(n)$ associated with the modular form  $f(z)=\sum_{n\geq1}\tau(n)q^n$. Let $r=rad(n)$ be the radical of $n\geq1$, and suppose that $\tau(r)\ne0$. Then, it satisfies the new explicit lower bound
\begin{equation} \label{eq5599.650}
\left |\tau\left (p^n\right )\right | \geq   n^{\frac{9}{2}+\frac{\log \log \log n}{\log \log n } },
\end{equation}
on a subset of integers $\mathcal{N}_f=\{n\in\N: \lambda(n)\ne0\}$ of density $\delta (\mathcal{N}_f)=1$.
}
\end{exa}

\section{Notation}\label{S9905}

The followings sets of numbers are used within.
\begin{enumerate} \label{eq9901.085}
\item $\tP=\{2,3,5,7, 11,13,\ldots\}$ denotes the set of prime integers,
\item $\N=\{0,1,2,3, \ldots\}$ denotes the set of nonnegative integers,
\item  $\Z=\{\-3,-2,-1,0,1,2,3, \ldots\}$ denotes the set of integers,
 \item $\Q=\{a/b:a,b\in \Z \}$ denotes the set of rational numbers,
 \item $\R$ denotes the set of real numbers,
\item $\C=\{s=a+ib: a,b \in \R\}$ denotes the set of complex numbers.\\
\end{enumerate}

For a pair of real valued functions $f:\R\longrightarrow \C$ and  $g:\R\longrightarrow \C$, the followings symbols are used without explanations.
\begin{enumerate} \label{eq9901.093}
\item $f\ll g$ implies that $af(x)\leq g(x)\leq bf(x)$, where $a>0$ and $b>0$ are constants, for all large real numbers $x$,
\item $f\gg g$ implies that $af(x)\geq g(x)\geq bf(x)$, where $a>0$ and $b>0$ are constants, for all large real numbers $x$,
\item $f=O(g)$ implies that $\vert f(x) \vert\leq \vert g(x)\vert$, for all large real numbers $x$,
\item $f=o(g)$ implies that $\vert f(x)/ g(x)\vert\to 0$ as $x\to \infty$,\\
\end{enumerate}

The followings arithmetic functions are used within.
\begin{enumerate} \label{eq9901.095}
\item $d(n)=\{d\mid n\}$, is the divisors counting function,\\
\item $\omega(n)=\{p\mid n\}$,  is the prime divisors counting function,\\
\item $\pi(x)=\{p\leq x\}$,  is the prime counting function,\\
\item $v_p(n)$, is the $p$-adic valuation, or e maximal prime power divisor of the integer $n$,\\
\item $rad(n)=\prod_{p\mid n}p$, is the radical of the integer $n$,
\end{enumerate}

\section{Lower Bound For All Large Integers}\label{S9955}
This result is based on an application of effective Diophantine approximation for algebraic numbers.
\begin{proof} {\bfseries (Theorem \ref{thm5599.200}) } Let $\alpha_p=p^{(k-1)/2}e^{i\theta_p}$, where $0\leq \theta_p\leq \pi$, be the root of the polynomial 
\begin{equation}\label{eq9955.100}
f(T)=a_2T+a_1T+a_0=T^2-\lambda(p) T+p^{k-1}
\end{equation}
of weight $k\geq4$. Modify the Binet formula (for integers sequences of the second order defined by a quadratic polynomial) into the following form
\begin{eqnarray} \label{eq9955.120}
\lambda(p^n)&=&\frac{\alpha_p^{n+1}-\overline{\alpha_p}^{n+1}}{\alpha_p-\overline{\alpha_p}}\\
&=&\frac{\alpha_p+\overline{\alpha_p}}{\alpha_p+\overline{\alpha_p}} \cdot\frac{\alpha_p^{n+1}-\overline{\alpha_p}^{n+1}}{\alpha_p-\overline{\alpha_p}}\nonumber\\
&=&\left (\alpha_p+\overline{\alpha_p}\right ) \cdot\frac{\alpha_p^{n+1}-\overline{\alpha_p}^{n+1}}{\alpha_p^2-\overline{\alpha_p}^2}\nonumber\\
&=&\left (\alpha_p+\overline{\alpha_p}\right ) \cdot\frac{\overline{\alpha_p}^{n+1}}{\alpha_p^2-\overline{\alpha_p}^2} \left (\alpha_p^{n+1} \overline{\alpha_p}^{-(n+1)}-1\right )\nonumber,
\end{eqnarray}

where $\overline{\alpha}_p$ is the conjugate of $\alpha_p$, and $p^n$ is a prime power. Replace the algebraic integer $\alpha_p=p^{(k-1)/2}e^{i\theta_p}$, where $0< \theta_p< \pi$. Taking absolute value and simplifying yield

\begin{eqnarray} \label{eq9977.082}
\left |\lambda(p^n)\right |&=&\left |p^{(k-1)/2}\cos(\theta_p)\right |
\left |\frac{p^{(k-1)(n+1)/2}e^{i\theta_p}}{2p^{(k-1)}\sin(2\theta_p)}\right |
\left | \alpha_p^{n+1} \overline{\alpha_p}^{-(n+1)}-1 \right | \\
&\geq& p^{(k-1)n/2}\left | \alpha_p^{n+1} \overline{\alpha_p}^{-(n+1)}-1 \right |\nonumber,
\end{eqnarray}
since $\lambda(p)\ne0$ implies $\theta\ne\pi/2$.
Replacing the estimate in Corollary \ref{cor2020.300} for $n\geq 1$, yields
\begin{eqnarray} \label{eq9977.084}
p^{(k-1)n/2}\left | \alpha_p^{n+1} \overline{\alpha_p}^{-(n+1)}-1 \right |
&\geq & p^{(k-1)n/2}\cdot \frac{1}{8p^{n+2k-2)}} \\
&\geq & \frac{1}{8}p^{(k-3)n/2-2k+2-\varepsilon} \nonumber,
\end{eqnarray}
where $\varepsilon>0$.
\end{proof}

\section{Lower Bound For Almost Every Integer}\label{S9911}
This result is based on an application of the current effective version of the Sato-Tate conjecture for modular forms, see \cite{TJ2021} for recent developments. Define the quantities
\begin{equation} \label{eq9911.013}
L_1=p^{\frac{k-1}{2}+\frac{\log \log p}{(\log p )^{1/2}} }\qquad \text{ and }\qquad U_1=p^{\frac{k-1}{2}+\varepsilon }.
\end{equation}
Given a modular form $f(z)=\sum_{n\geq1}\lambda(n)q^n$ of weight $k\geq1$, let
\begin{equation}\label{eq9911.0500}
\mathcal{P}_f=\left\{p\in\tP: \lambda(p)\geq L_1\right\}\subset \tP
\end{equation}
be a subset of primes of density $\delta \left(\mathcal{P}_f \right) \geq 0$, and let 
\begin{equation}\label{eq9911.0510}
\mathcal{N}_f=\{n=\prod_{p \mid n} p^{v_p}\in\N: p\in \mathcal{P}_f\}\subset \N,
\end{equation}
where $v_p(n)\mid \mid n$ is the maximal prime power divisor, be the generated multiplicative subset of integers of density $\delta(\mathcal{N}_f)\geq0$.

\begin{lem} \label{lem9911.444} The multiplicative subset of integers  $\mathcal{N}_f \subset \N$ has density $\delta(\mathcal{N}_f)=1$.

\end{lem}

\begin{proof} By Theorem 1 in \cite{GT2020}, the inequality
\begin{equation} \label{eq9911.110}
p^{(k-1)/2}\frac{\log \log p}{(\log p )^{1/2} }< \left|\lambda\left (p\right )\right |<2p^{(k-1)/2},
\end{equation}
is valid on a subset of primes $\mathcal{P}_f=\{p\in\tP: \lambda(p)\geq L_1\}$ of density $1$. Therefore, the number of integers generated by the subset of primes $\mathcal{P}_f$ has the asymptotic counting function
\begin{eqnarray} \label{eq9911.130}
N_f(x)&=&\sum_{n\leq x}\chi_f\left (n\right )\\
&= & C_f\cdot \frac{x}{\log x} 
\prod_{\substack{p\leq x\\ p \in\mathcal{P}_f}}\left ( 1+\frac{\chi_f\left (p\right )}{p}+\frac{\chi_f\left (p^2\right )}{p^2}+\cdots \right )\nonumber,
\end{eqnarray}
where the constant
\begin{equation} \label{eq9911.125}
C_f= \frac{1}{e^{\gamma \tau }\Gamma (\tau )}+o(1)=\frac{1}{e^{\gamma }}+o(1),
\end{equation}
where $\tau=1$ is the density of the subset of primes $\mathcal{P}_f$, $\gamma$ is Euler constant, and $\Gamma (s)$ is the gamma function, see Lemma \ref{lem9933.150}, and Lemma \ref{lem9988.700}. Now, applying Mertens theorem, see  Lemma \ref{lem9977.222}, returns

\begin{eqnarray} \label{eq9911.170}
N_f(x)&=&\left (  \frac{1}{e^{\gamma }}+o(1)\right )\cdot  \frac{x}{\log x}\prod_{\substack{p\leq x\\ p \in\mathcal{P}_f}}\left ( 1+\frac{1}{p}+\frac{1}{p^2}+\cdots \right )\\
&=&\left (  \frac{1}{e^{\gamma }}+o(1)\right )\cdot  \frac{x}{\log x}\prod_{\substack{p\leq x\\ p\in \mathcal{P}_f}}\left ( 1-\frac{1}{p} \right )^{-1}\nonumber\\
&=&\left (  \frac{1}{e^{\gamma }}+o(1)\right )\cdot  \frac{x}{\log x}\left ( \left( e^{\gamma }\log x\right )^{\tau} +O\left (\frac{1}{\log x} \right )\right )\nonumber\\
&=&x +O\left (\frac{x}{\log^2 x}\right )\nonumber,
\end{eqnarray}
as $x\to\infty$.
\end{proof}

For sufficiently large integers $n\geq1$, the estimates
\begin{align} \label{eq9911.095}
    d(n)&=O(n^{\varepsilon}), \\
    \omega(n)&\leq(1+\varepsilon)\log n/\log \log n \nonumber,
\end{align}
where $\varepsilon>0$ is a small number, are used in the proof below, see \cite[Theorem 13.12]{AT1976} for more details. 

\begin{proof} {\bfseries (Theorem \ref{thm9911.200})} The density $1$ claim for the subset of integers 
\begin{equation} \label{eq9911.105}
\mathcal{N}_f=\{n\in\N: p\mid n \Rightarrow p\in \mathcal{P}_f\}
\end{equation}
is proved in Lemma \ref{lem9911.444}. Next, to estimate the lower bound on the subset of integers $
\mathcal{N}_f$, consider the prime decomposition
\begin{equation}\label{eq9911.0200}
\left|\lambda\left (n\right )\right |=\prod_{p^{v_p} \mid \mid n}\left|\lambda\left (p^{v_p}\right ) \right | 
=\prod_{\substack{p^{v_p} \mid \mid n\\ p\notin \mathcal{P}_f}}\left|\lambda\left (p^{v_p}\right ) \right |  \prod_{\substack{p^{v_p} \mid \mid n\\ p\in \mathcal{P}_f}}\left|\lambda\left (p^{v_p}\right ) \right |.
\end{equation}
Replacing \eqref{eq9911.110} into \eqref{eq9911.0200} returns
\begin{eqnarray}\label{eq9911.0610}
\left|\lambda\left (n\right )\right |
&\geq&\prod_{\substack{p^{v_p} \mid \mid n\\ p\notin \mathcal{P}_f}}\left|\lambda\left (p^{v_p}\right ) \right |.   \prod_{\substack{p^{v_p} \mid \mid n\\ p\in \mathcal{P}_f}} p^{\frac{k-1}{2}v_p}\cdot \frac{\log \log p}{(\log p)^{1/2}}\\
&\geq& n^{\frac{k-1}{2}}\prod_{\substack{p \mid n\\ p\in \mathcal{P}_f}}\frac{\log \log p}{(\log p)^{1/2}}\nonumber\\
&\geq& n^{\frac{k-1}{2}}\prod_{\substack{p \mid n\\ p\in \mathcal{P}_f}}\frac{\log \log n}{\log n}\nonumber\\
&\geq& n^{\frac{k-1}{2}}\left(\frac{\log \log n}{ \log n}\right )^{(1+\varepsilon)\frac{\log n}{ \log \log n}}\nonumber\\
&\geq&n^{\frac{k-3}{2}+\frac{\log \log \log n}{\log \log n}  }\nonumber,
\end{eqnarray}
where $\omega(n)\leq(1+\varepsilon)\log n/\log \log n $ is an upper bound for the number of prime factors from the subset of primes 
$\mathcal{P}_f=\{p\in\tP: |\lambda(p)|\geq L_1\}$ of density $1$. The lower bound \eqref{eq9911.0610} is valid for all large numbers $n \geq 1$. Furthermore, the upper bound follows from the first relation in \eqref{eq9911.095}. \\

Hence, almost every integer $n\geq1$, the $n$th Fourier coefficient satisfies the inequalities
\begin{equation} \label{eq9911.0800}
n^{\frac{k-3}{2}+\frac{\log \log \log n}{\log \log n  }}< \left|\lambda\left (n\right )\right | <n^{\frac{k-1}{2}+\varepsilon },
\end{equation}
as $n \to \infty$.
\end{proof}

\begin{exa}{\normalfont For $k=12$, the $n$th Fourier coefficient is given by the Ramanujan tau function $\tau(n)$. At $n=1000000$, the values specified by Theorem \ref{thm9911.200} are the followings.
\begin{enumerate}
 \item $\displaystyle n^{\frac{k-3}{2}+\frac{\log \log \log n}{\log \log n  }}=1.60\times10^{29}$,\\
 
    \item $\displaystyle \tau(10^{6})=262191418612588689102548992000000=2.62\times10^{ 32}$,\\
    
     \item $\displaystyle d(n)n^{\frac{k-1}{2} }=4.90\times10^{34}$.
\end{enumerate}


}
\end{exa}
\section{Basic Results In Diophantine Approximations } \label{S2020}
The concept of measures of irrationality of real numbers is discussed in \cite[p.\ 556]{WM2000}, \cite[Chapter 11]{BB1987}, et alii. This concept can be approached from several points of views. 

\begin{dfn} \label{dfn2020.01} {\normalfont The irrationality measure $\mu(\alpha)$ of a real number $\alpha \in \R$ is the infimum of the subset of  real numbers $\mu(\alpha)\geq1$ for which the Diophantine inequality
\begin{equation} \label{eq2020.035}
  \left | \alpha-\frac{p}{q} \right | \gg\frac{1}{q^{\mu(\alpha)+\varepsilon} }
\end{equation}
where $\varepsilon >0$ is an arbitrary small number, holds for all large $q \geq 1$.
}
\end{dfn}

Let $\alpha $ be a root of an irreducible polynomial $f(x)\in \Z[x]$ of degree $\deg f=d$. The Liouville inequality 
\begin{equation}\label{eq2020.090}
\left |\alpha-\frac{p}{q}\right |>\frac{c}{q^d},
\end{equation}
where $c>0$ is a constant, is the oldest result for irrationality measure of an irrational number as $\alpha $. The Liouville result was superseded by Thue-Roth theorem for the approximations of algebraic integers 
\begin{equation}\label{eq2020.095}
\left |\alpha-\frac{p}{q}\right |>\frac{c(\alpha,\varepsilon)}{q^{2+\varepsilon}},
\end{equation}
where the constant $c(\alpha,\varepsilon)>0$ is not computable.\\

A completely explicit version of the Liouville theorem is consider here. The proof given below is relatively recent, an older version appears in \cite[Theorem 7.8]{NI1956}. The preliminary definitions are a few concepts used in the proof.

\begin{dfn} \label{dfn2020.100}{\normalfont Let $r=a/b\in \Q^{\times}$ with $\gcd(a,b)=1$, be a rational number. The \textit{height} is defined by
\begin{equation} \label{eq2020.100}
H(r)=\max \{|a|,|b|\}.
\end{equation} 
}
\end{dfn}

\begin{dfn} \label{dfn2020.110} {\normalfont Let $\alpha $ be a root of the irreducible polynomial 
\begin{equation} \label{eq2020.105}
f(x)=a_d\left (x-\alpha_1\right )\left (x-\alpha_2\right )\cdots\left (x-\alpha_d\right ).
\end{equation}  
The \textit{Mahler height} is defined by
\begin{equation} \label{eq2020.110}
H(\alpha)=|a_d| \prod_{1\leq i\leq d}\max \{1,|\alpha_i|\}.
\end{equation} 
}
\end{dfn}
 
\begin{thm}{\normalfont (Liouville)} \label{thm2020.200} Let $\alpha$ be a real algebraic number of degree $d\geq 1$. There is a constant $c(\alpha)>0$ depending only on $\alpha$ such that
\begin{equation}\label{eq2020.115}
\left |\alpha-\frac{p}{q}\right| \geq\frac{c(\alpha)}{H(p/q)^{d}},
\end{equation}
where $c(\alpha)\geq 2^{1-d}H(\alpha)^{-1}$, for every rational number $p/q$ if $d\geq 2$.
\end{thm}
\begin{proof} (Same as \cite[Theorem 8.1.1]{EJ2007}) Let $f(x)=a_dx^d+a_{d-1}x^{d-1}+\cdots+a_1x+a_0\in \Z[x]$ be the minimal polynomial of $\alpha$. Take $p/q\in \Q$, where $p,q\in \Z$, and consider the number
\begin{equation}\label{eq2020.120}
F(p,q)=q^df(p/q)=a_dx^d+a_{d-1}x^{d-1}q+\cdots+a_1xq^{d-1}+a_0q^{d}.
\end{equation}
By assumption, $f(p/q)\ne0$. Hence, $F(p,q)$ is a nonzero integer. So $|F(p,q)|\geq1$. To obtain an upper bound for $|F(p,q)|$, rewrite it in factored form
\begin{equation} \label{eq2020.125}
F(p,q)=a_d\left (p-\alpha_1q\right )\left (p-\alpha_2q\right )\cdots\left (p-\alpha_dq\right ),
\end{equation} 
and estimate each factor as follows. Let $\alpha=\alpha_1$, then
\begin{enumerate} 
\item $ \displaystyle 
\left |p-\alpha q\right |=\left |\frac{p}{q}-\alpha\right | \left |q\right |\leq \left |\frac{p}{q}-\alpha\right |H(\alpha),
$
\item $ \displaystyle 
\left |p-\alpha_i q\right |\leq 2\max \left \{1,|\alpha_i|\right \} \cdot \max \left \{ |p|, |q|\right \}
\leq 2\max \left \{1,|\alpha_i|\right \} H(p/q),
$
\end{enumerate}
where $i=2,3,\ldots, d$. Thus,
\begin{eqnarray}\label{eq2020.130}
\left |F(p,q)\right |&=&|a_d|\prod_{1\leq i\leq d}\left |p-\alpha_iq\right |\\
&\leq &|a_d|\left |\frac{p}{q}-\alpha\right | \cdot H(p/q)^d\cdot 2^{d-1}\cdot \prod_{2\leq i\leq d}\max \left \{1,|\alpha_i|\right \}\nonumber\\
&\leq &|a_d|\left |\frac{p}{q}-\alpha\right | \cdot H(p/q)^d\cdot 2^{d-1}\cdot H(\alpha)\nonumber.
\end{eqnarray}
Combined the lower bound $|F(p,q)|\geq1$ and the inequality \eqref{eq2020.130} to complete the proof.
\end{proof}

\begin{cor} \label{cor2020.200} Let $\beta$ be an algebraic number of degree $d=2$ and height $H(\beta)\leq 4p^{2(k-1)}$. Let $\{p_m/q_m:m\geq1\}$ be a subsequence of convergents of height $H(p_m/q_m)\leq p^{n/2}$, where $p^n$ is a large prime power. Then,
\begin{equation}\label{eq2020.205}
\left |\beta-\frac{p_m}{q_m}\right| \geq\frac{1}{8p^{n+2k-2}},
\end{equation}
as $q_m\leq p^{n/2} \to \infty.$
\end{cor}
\begin{proof} Let $\beta$ be a root of the polynomial $f(T)=a_2T^{2}+a_1T+a_0$ of degree $\deg f=2$, and coefficients $|a_0|,|a_1|\leq p^{k-1}$, $a_2=1$. These data imply that $|\beta|\leq 2p^{k-1}$, and the height is at most
\begin{equation} \label{eq2020.210}
H(\beta)=|a_d| \prod_{1\leq i\leq d}\max \{1,|\beta_i|\} \leq 4p^{2(k-1},
\end{equation}
see Definition \ref{dfn2020.110}. Hence, the constant in inequality \eqref{eq2020.105} has at least the value
\begin{equation}\label{eq2020.215}
c(\beta)\geq 2^{1-d}H(\beta)^{-1}= 2^{1-2}\cdot (4p^{2(k-1)})^{-1}=\frac{1}{8p^{2(k-1)}}
\end{equation}
and the height of the rational approximation is 
\begin{equation}\label{eq2020.220}
H(p_m/q_m)\leq p^{n/2}.
\end{equation}
An application of Theorem \ref{thm2020.200} yields
\begin{eqnarray}\label{eq2020.225}
\left |\beta-\frac{p_m}{q_m}\right| 
&\geq &\frac{c(\beta)}{H(p_m/q_m)^{d}}\\
&\geq & \frac{1}{8p^{2(k-1)}}\frac{1}{\left(p^{n/2}\right )^2}\nonumber\\
&\geq &\frac{1}{8p^{n+2k-2}} \nonumber.
\end{eqnarray}
\end{proof}

\begin{cor} \label{cor2020.300} Let $\beta=\alpha_p^{n+1} \overline{\alpha_p}^{-(n+1)}$ be an algebraic number of degree $d=2$ and height $H(\beta)\leq 4p^{2(k-1)}$. Let $\{p_m/q_m:m\geq1\}$ be a subsequence of convergents of height $H(p_m/q_m)\leq p^{n/2}$, where $p^n$ is a large prime power. Then,
\begin{equation}\label{eq2020.300}
\left |\beta-1\right |>\frac{1}{8p^{n+2k-2}},
\end{equation}
as $q_{m}\leq p^{n/2} \to \infty.$
\end{cor}
\begin{proof} The inverse triangle inequality, and Corollary \ref{cor2020.200} lead to
\begin{eqnarray}\label{eq2020.310}
\left| \beta-1\right | &=&\left| \beta-1+ \frac{p_m}{q_m}-\frac{p_m}{q_m}\right |\\
&\geq&\left| \left |\beta- \frac{p_m}{q_m}\right |- \left| 1-\frac{p_m}{q_m}\right |\right |\nonumber \\
&\geq& \left |\beta- \frac{p_m}{q_m}\right |\nonumber \\
&\geq&\frac{1}{8p^{n+2k-2}} \nonumber,
\end{eqnarray}
since
\begin{equation}\label{eq2020.320}
\left| 1-\frac{p_m}{q_m}\right |=\left |\frac{q_m-p_m}{q_m}\right |\geq \frac{1}{p^{n/2}},
\end{equation}
as $q_{m}\leq p^{n/2} \to \infty.$
\end{proof}

\section{Coefficients Characteristic Function}\label{S9933}

The characteristic function of the nonzero Fourier coefficients of a modular form $f(z)=\sum_{n\geq1}\lambda(n)q^n$ is defined by
\begin{equation}\label{eq9933.100}
\chi_{f}(p)=
\begin{cases}
1& \mbox{ if } \lambda(p)\ne 0;\\
0& \mbox{ if } \lambda(p)= 0.
\end{cases}
\end{equation}

\begin{lem} \label{lem9933.150} The function $\chi_{f}:\Z\longrightarrow 
\{0,1\}$ is completely multiplicative.
\end{lem}
\begin{proof} Suppose that $\lambda(p)\ne0$, equivalently $\chi_{f}\left (p^{}\right )=1$. It is sufficient to verify the multiplicative property at the prime power $p^k$, where $k\geq2$. Recursively eliminating the term $\chi_{f}\left (p^{}\right )=1$ yields
\begin{eqnarray}\label{eq9933.150}
\chi_{f}\left (p^{k}\right )
&=&\chi_{f}\left (p^{k-1}\right )\chi_{f}\left (p^{}\right )\\
&=&\chi_{f}\left (p^{k-1}\right )\nonumber\\
&=&\chi_{f}\left (p^{k-2}\right )\chi_{f}\left (p^{}\right )\nonumber\\
&\vdots& \qquad \qquad   \vdots\qquad \qquad  \vdots\qquad \qquad    \nonumber\\
&=&\chi_{f}\left (p^{}\right )\nonumber\\
&=&1\nonumber.
\end{eqnarray}
Therefore, $1=\chi_{f}\left (p^{}\right )=\chi_{f}\left (p^{2}\right )=\cdots =\chi_{f}\left (p^{k}\right )$, as claimed.
\end{proof}

\section{Wirsing Formula} \label{S9988}
This formula provides decompositions of some summatory multiplicative functions as products over the primes supports of the functions. This technique works well with certain multiplicative functions, which have supports on subsets of primes numbers of nonzero densities. \\

\begin{lem} {\normalfont (\cite[p. 71]{WE1961})} \label{lem9988.700} Suppose that \(f:\mathbb{N}\longrightarrow \mathbb{C}\) is a multiplicative function with the following properties.
\begin{enumerate} [font=\normalfont, label=(\roman*)]
\item $f(n) \geq 0$ for all integers $n\in \mathbb{N}$. 
\item $f\left(p^k\right)\leq c^k$ for all integers $k\in \mathbb{N}$, and $c<2$ constant. 
\item There is a constant $\tau >0$ such 
\begin{equation}\label{eq9988.700}
\sum _{p\leq x} f(p)=(\tau +o(1)) \frac{x}{\log  x}
\end{equation} 
as $x \longrightarrow  \infty$.  
\end{enumerate}

Then
\begin{equation}\label{eq9988.710}
	\sum _{n\leq x} f(n)=\left(\frac{1}{e^{\gamma \tau }\Gamma (\tau )}+o(1)\right)\frac{x}{\log  x}\prod _{p\leq x} \left(1+\frac{f(p)}{p}+\frac{f\left(p^2\right)}{p^2}+\cdots
	\right) .
	\end{equation}
	
\end{lem}

The gamma function appearing in the above formula is defined by \(\Gamma (s)=\int _0^{\infty }t^{s-1}e^{-s t}d t, s\in \mathbb{C}\).

\section{Harmonic Sums And Products And Positive Densities}\label{S9977}
Let $f(z)=\sum_{n\geq1}\lambda(n)q^n$ be a modular form of weight $k\geq1$, let $L_1=L_1(p)>0$ be a decreasing function of the primes $p$, and let \(\mathcal{P}_f=\left\{ p\in \mathbb{P} :|\lambda(n)|\geq L_1\right\}\subset \mathbb{P}\) be a subset of primes of nonzero density \(\delta \left(\mathcal{P}_f\right)>0\). The proofs of these results are based on standard analytic number theory methods in the literature.

\begin{lem} \label{lem9977.111} If \(x\geq 1\) is a large number, and $\tau=\delta \left(\mathcal{P}_f\right)>0$ is the density of the subset of primes, then, 
\begin{equation}\label{eq9977.111}
 \sum_{\substack{p\leq x \\  p \in \mathcal{P}_f}} \frac{1}{p}=\left(  \log \log
	x+B\right)\tau + O\left(\frac{1}{\log  x} \right) ,   
\end{equation}
where $B>0$ is a constant
\end{lem}

\begin{proof} Let \(\pi _f(x)=\#\left\{ p\leq x:\vert \lambda(n)\vert \;\geq L_1 \right\}=\tau\pi (x)\) be the counting measure of the
	corresponding subset of primes \(\mathcal{P}_f\). To estimate the asymptotic order of the prime harmonic sum, use the Stieltjes integral representation:
	\begin{equation}\label{eq9977.121}
		\sum _{\substack{p\leq x \\ p\in \mathcal{P}_f}} \frac{1}{p} =\int_{1}^{x}\frac{1}{t}d \pi _f(t) =\frac{\pi _f(x)}{x}
		+\int_{1}^{x}\frac{\pi _f(t)}{t^2}d t.
	\end{equation} 
	Applying the prime number theorem $\pi(x)=x/\log x +O(x/\log ^2 x)$, see \cite[Section 27.12]{DLMF}, \cite{EL1985}, etc., yields
	\begin{eqnarray}\label{eq9977.131}
		\int_{1}^{x}\frac{1}{t}d \pi _f(t) 
		&=&\frac{\tau }{\log  x}+O\left(\frac{1}{\log ^2x}\right)\nonumber\\
		&&\hskip 1.0 in+\;\tau\int_{1}^{x}\left(\frac{1}{t \log  t}+O\left(\frac{1}{t \log ^2t}\right)\right)d t \\
		&=&\tau \left (\log \log
		x+B\right)+O\left(\frac{1}{\log  x}\right)  \nonumber,
	\end{eqnarray} 
	where $B>0$ is Mertens constant.
\end{proof}

The Euler constant and the Mertens constant are defined by the limits 
\begin{equation} \label{eq9977.141}
\gamma =\lim_{x \rightarrow  \infty }
	\left(\sum _{ p\leq x } \frac{\log  p}{p-1}-\log  x\right) \qquad\text{ and } \qquad B=\lim_{x \rightarrow  \infty } \left(\sum _{ p\leq x } \frac{1}{p}-\log \log x\right),
\end{equation}
or some other equivalent definitions, respectively. \\

\begin{lem} \label{lem9977.222} If \(x\geq 1\) is a large number, and $\mathcal{P}_f\subset \tP$ is a subset of primes of nonzero density $\tau=\delta \left(\mathcal{P}_f\right)>0$, then, 
\begin{equation}\label{eq9977.222}
 \prod _{\substack{ p\leq x\\ p\in \mathcal{P}_f}} \left(1-\frac{1}{p}\right)^{-1}=\left (e^{\gamma
		} \log x\right )^{\tau}+O\left(\frac{1}{\log  x}\right) . 
\end{equation}
\end{lem}

\begin{proof} Express the logarithm of the product as 
	\begin{equation}
		\sum _{\substack{p\leq x\\ p\in \mathcal{P}_f}} \log \left(1-\frac{1}{p}\right)^{-1}=\sum
		_{\substack{p\leq x\\ p\in \mathcal{P}_f}} \sum _{k\geq 1} \frac{1}{k p^k}=\sum _{\substack{p\leq x\\ p\in \mathcal{P}_f}} \frac{1}{p} +\sum _{\substack{p\leq x\\ p\in \mathcal{P}_f}}
		\sum _{k\geq 2} \frac{1}{k p^k} .
	\end{equation}   
Apply Lemma \ref{lem9977.111}, and the linear independence relation $B =\gamma -\sum _{p\geq 2} \sum _{k\geq 2} \left(k p^k\right) ^{-1}$, see \cite[Theorem 427]{HW2008}, to complete the verification.
\end{proof}


\currfilename.\\

\end{document}